\documentclass{article}
\usepackage {amssymb}
\usepackage{mathrsfs}
\usepackage {amsmath}
\usepackage {amsthm}
\usepackage{graphicx}
\usepackage {amscd}
\usepackage{subfigure}

\usepackage{xypic}
\usepackage[colorlinks, citecolor=red]{hyperref}
\newcommand{\PZI}{{\frac{\partial}{\partial z^i}}}
\newcommand{\PZIU}{{\frac{\partial u}{\partial z^i}}}
\newcommand{\PBZJ}{{\frac{\partial}{\partial\bar{z}^j}}}

\newcommand{\SF}{{\mathscr{F}}}
\newcommand{\SFPR}{{\mathscr{F}_{r_{i\bar{j}}}}}

\newtheorem{thm}{Theorem}
\newtheorem{prop}{Proposition}
\newtheorem{defn}{Definition}

\newtheorem{rem}{Remark}

\newtheorem{lem}{Lemma}

\begin{document}
\title{A Poho\v{z}aev identity and critical exponents of some complex Hessian equations}
\author{Chi Li}

\maketitle

\begin{abstract}
In this note, we prove some non-existence results
for Dirichlet problems of complex Hessian equations. The non-existence results are proved using the Poho\v{z}aev method. We also prove existence results for radially symmetric solutions. The main difference of the complex case with the real case is that we don't know if a priori radially symmetric property holds in the complex case.
\end{abstract}
 
\section{Introduction}

In \cite{Tso1}, Tso considered the following real k-Hessian equation:
\begin{equation}\label{rhess}
S_k(u_{\alpha\beta})=(-u)^p \mbox{ on } \Omega, \quad u=0 \mbox{ on } \partial\Omega.  
\end{equation}
Tso proved the following result
\begin{thm}[\cite{Tso1}]
Let $\Omega$ be a ball and $\tilde{\gamma}(k,d)=\left\{\begin{array}{cc}\frac{(d+2)k}{d-2k}& 1\le k<\frac{d}{2}\\ \infty&\frac{d}{2}\le k<d\end{array}\right.$ Then
(i) \eqref{rhess} has no negative solution in $C^1(\bar{\Omega})\cap C^4(\Omega)$ when $p\ge \tilde{\gamma}(k,d)$; (ii) It admits a negative solution which is radially symmetric and is in $C^2(\bar{\Omega})$ when $0<p<\tilde{\gamma}(k,d)$, $p$ is not equal to $k$.
\end{thm}
The non-existence result above was proved by the Poho\v{z}aev method.  In this article, we first generalize Tso's result to case of complex k-Hessian equation. From now on, let $B_R$ be the ball of radius $R$ in $\mathbb{C}^n$. We first consider the following equation
\begin{equation}\label{chess}
S_k(u_{i\bar{j}})=(-u)^p \mbox{ on } B_R, \quad u=0 \mbox{ on } \partial B_R.
\end{equation}
where the complex k-Hessian operator is defined as 
\[
S_k(u_{i\bar{j}})=\frac{1}{k!}\sum_{1\le i_1, \dots, j_k\le n}\delta^{i_1\dots i_k}_{j_1\dots j_k}u_{i_1\bar{j}_1} \dots u_{i_k\bar{j}_k}.
\]
Our first result is
\begin{thm}\label{critCHE}
Define $\gamma(k,n)=\frac{(n+1)k}{n-k}=\tilde{\gamma}(k,2n)$. Then (i) \eqref{chess} has no nontrivial nonpositive solution in $C^2(\bar{B}_R)\cap C^4(B_R)$ when $p\ge \gamma(k,n)$; (ii) It admits a negative solution which is 
radially symmetric and is in $C^2(\bar{B}_R)$ when $0<p<\gamma(k,n)$ and $p$ is not equal to $k$.
\end{thm}
\begin{rem}
By scaling, we get solution to $S_k(u_{i\bar{j}})=\lambda (-u)^p$ for any $\lambda>0$ if $p$ satisfies the restrictions. When $p=k$, we are in the eigenvalue problem, as in the real Hessian case (\cite{Wan}), one should be able to show that there exists a $\lambda_1>0$ such that there is a nontrivial nonpositive solution to the equation: $S_k(u_{i\bar{j}})=\lambda_1(-u)^k$. Moreover, the solution is unique up to scaling.  This will be discussed elsewhere.
\end{rem}
\begin{rem}
By \cite{GNN}, \cite{Del}, the solution to \eqref{rhess} is a priori radially symmetric. However, it's not known if all the solution to \eqref{chess} are radially symmetric. The classical moving plane method for proving radial symmetry works for large classes real elliptic equations but doesn't seem to work in the complex case (cf. \cite{Del}). For the recent study of complex Hessian equations, see \cite{Blo}, \cite{Lisy}, \cite{Zhou} and the reference therein.
\end{rem}

Next we use Poho\v{z}aev method to prove a non-existence result for the following equation:
\begin{equation}\label{cHLBG}
S_k(u_{l\bar{m}})=a\frac{e^{-u}}{\int_{B_1}e^{-u}dV} \mbox{ on } B_1 , \quad u=0 \mbox{ on } \partial B_1.
\end{equation}
Note that when $k=n$, we have a Monge-Amp\`{e}re equation:
\begin{equation}\label{cMALBG}
\det(u_{l\bar{m}})=a\frac{e^{-u}}{\int_{B_1}e^{-u}dV} \mbox{ on } B_1 , \quad u=0 \mbox{ on } \partial B_1.
\end{equation}
Since the domain we consider is the unit ball, there are natural solutions to \eqref{cMALBG} coming from potential of Fubini-Study metric on $\mathbb{P}^n$:
\begin{equation}\label{FSsol}
u_{\epsilon}=(n+1)[\log(|z|^2+\epsilon^2)-\log (1+\epsilon^2)],
\end{equation}
with the parameter $a$ in \eqref{cMALBG} being
\begin{equation}\label{aepsilon}
a_\epsilon=(n+1)^n\epsilon^2 \int_0^1\frac{r^{2n-1}dr}{(r^2+\epsilon^2)^{n+1}}\omega_{2n-1}=(n+1)^n\omega_{2n-1}\int_0^{1/\epsilon}\frac{t^{2n-1}dt}{(1+t^2)^{n+1}}=(n+1)^n\frac{\omega_{2n-1}}{2(1+\epsilon^2)^n}.
\end{equation}
where we will use $\omega_{d-1}=\frac{2\pi^{d/2}}{\Gamma(d/2)}$ to denote the volume of the (d-1)-dimensional unit sphere $S^{d-1}$. In particular $\omega_{2n-1}=\frac{2\pi^n}{(n-1)!}$. So we get that when 
\begin{equation}\label{critLBG}
0<a<a_{0}=(n+1)^n\frac{\pi^n}{n!},
\end{equation}
there exists radially symmetric solutions for \eqref{cMALBG}.  Again it's an open question (\cite{Del}, \cite{BeBe}) whether all
solutions to \eqref{cMALBG} are a priori radially symmetric, which would imply \eqref{FSsol} gives all the solutions to \eqref{cMALBG}.  Without a priori
radially symmetric properties, we can still use Poho\v{z}aev method to get
\begin{thm}\label{nonlocalthm}
For the Dirichlet problem \eqref{cHLBG}, there exists $\alpha(k,n)>0$ such that there exists no solution to \eqref{cHLBG} in $C^2(\bar{B}_1)\cap C^4(B_1)$ when $a>\alpha(k,n)$. Moreover, when $k=n$, we can make $\alpha(n,n)=a_0=(n+1)^n\frac{\pi^n}{n!}$ and  \eqref{cMALBG} has no solution in $C^2(\bar{B}_1)\cap C^4(B_1)$ if $a\ge a_0$. In other words, the $a_0$ in \eqref{critLBG} is sharp and can not be obtained, at least for solutions with enough regularity.
\end{thm}
\begin{rem}
\eqref{cMALBG} was a local version of K\"{a}hler-Einstein metric equation. It was extensively studied in \cite{BeBe} for even general hyperconvex domains. Note that, the normalization here differs from that in \cite{BeBe} by a factor of $\pi^n/n!$. Berman-Berndtsson proved that equation \eqref{cHLBG} has a solution when $a<a_{0}$ on any hyperconvex domain which actually is a global minimizer of a functional associated to Moser-Trudinger-Onofri inequality. However, it's not known if there are solutions when $a\ge a_{0}$. Our observation is that, 
the Poho\v{z}aev method used in \cite{CLMP} for Laplace equation can be generalized and gives nonexistence results for star-shaped and (strongly) k-pseudoconvex domains. For simplicity we restrict to the ball to state our result. Note that, when $n=1$, $a_0=2\pi$. Since $\Delta=4 u_{z\bar{z}}$ on the complex plane, this is the well known result for Laplace equation of type \eqref{cHLBG} (\cite{CLMP}).
\end{rem}
In the last part, we will restrict ourselves to radially symmetric solutions. Radial symmetry reduces the equation \eqref{cHLBG} to the following equation. 
\begin{equation}\label{radreduction}
(u_s^ks^n)_ss^{1-n}=A(k,n)^{-1} \frac{a e^{-u}}{\int_0^1 e^{-u(s)}s^{n-1}ds}, \quad u=0 \mbox{ on } \partial B_1.\quad A(k,n)=\frac{\omega_{2n-1}}{2k}\binom{n-1}{k-1}.
\end{equation}
See equation \eqref{radLBGEQ}. Using phase plane method, we will prove the following result.
\begin{thm}\label{RadExist}
Define $\beta(k,n)=k^{k-1}\binom{n-1}{k-1}\frac{\pi^n}{(n-1)!}$. We have the following description of solutions of \eqref{radreduction}, or equivalently the radially symmetric solutions
of \eqref{cHLBG}.
\begin{enumerate}
\item
There exists $\alpha^*(k,n)$ such that 
\begin{enumerate}
\item $k<n$, \eqref{radreduction} admits a solution if and only if $a\le \alpha^*(k,n)$. Moreover, $\alpha^*(k,n)=\beta(k,n)$ if $n-k\ge 4$.
\item
When $k=n$, \eqref{radreduction} admits a solution if and only if $a<\alpha^*(n,n)=(n+1)^n\frac{\pi^n}{n!}$.
\end{enumerate} 
\item
$0<n-k<4$. The solutions to \eqref{radreduction} are unique for small $a>0$. When $a=\beta(k,n)$, there exist infinitely many 
solutions for \eqref{radreduction}. When $\alpha^*(k,n)\ge a\neq \beta(k,n)$, there exists finitely many solutions to \eqref{radreduction}. Moreover, the number of solutions
tends to infinity as $a$ approaches $\beta(k,n)$.
\item
When $n=k$ or $n-k\ge 4$. For every $a>0$, there exists at most one solution of \eqref{radreduction}. 
\end{enumerate}
\end{thm}
Similar radially symmetric problems for real equations were considered before by several 
people (\cite{JoLu}, \cite{BHN}, \cite{JaSc}). They all used the phase plane method initiated in \cite{JoLu}. The above theorem  
generalizes \cite[Theorem 1]{BHN} to the complex Hessian case.  This is achieved by generalizing and modifying the argument used in \cite{BHN}. See also Remark \ref{regionmodify}.

\section{A Poho\v{z}aev identity for complex Monge-Amp\`{e}re equation}
In \cite{Poh} Poho\v{z}aev established an identity for solutions of the Dirichlet problem
\begin{equation}\label{linear}
\Delta u+f(u)=0 \mbox{ in } \Omega, \quad u=0 \mbox{ on } \partial\Omega.
\end{equation}
He used this identity to show that the problem \eqref{linear} has no nontrivial solutions when $\Omega$ is a bounded star-shaped domain in 
$\mathbb{R}^d$ and $f=f(u)$ is a continuous function on $\mathbb{R}$ satisfying the condition
\[
(d-2)uf(u)-2d F(u)>0 \quad \mbox{ for } u\neq 0, 
\]
where $F$ denotes the primitive $F(u)=\int_0^u f(t)dt$ of $f$. Later, Pucci-Serrin \cite{PuSe} generalized Poho\v{z}aev identity to identities for much general variational equations, and they 
obtained non-existence results using these type of identities.  We will follow Pucci-Serrin to derive a Poho\v{z}aev identity in the complex case. We will consider the general variational problem
associated to the functional
\[
\mathbb{F}=\int_\Omega \SF(z,u(z), u_{i\bar{j}}(z))dV.
\]
It's easy to verify that the Euler-Lagrange equation for $\mathbb{F}$ is 
\begin{equation}\label{ELeq}
\frac{\partial^2}{\partial z^i\partial\bar{z}^j}\mathscr{F}_{r_{i\bar{j}}}+\mathscr{F}_u=0.
\end{equation}

We can now state the Poho\v{z}aev type identity we need. Note that the coefficient for the last term is slightly different with the formula in \cite[(29)]{PuSe} in the real case.
\begin{prop}
For any constant $c$,
\begin{eqnarray}\label{PohId}
&&\PZI\left(z^i\SF+ (cu+z^q u_q)\PBZJ\SFPR\right)-\PBZJ\left(\PZI(cu+z^q u_q)\SFPR\right)\nonumber\\
&=&n\SF+z^i\SF_{z^i}-cu\SF_u-(c+1)u_{i\bar{j}}\SFPR.
\end{eqnarray}
\end{prop}
\begin{proof}
This follows from direct computation. We give some key steps in the calculation. 
\begin{itemize}
\item
Multiply $u$ on both sides of equation \eqref{ELeq} and use the product rule for differentiation we get:
\begin{equation}\label{calstep1}
\PZI (u\PBZJ\SFPR) - \PBZJ(\PZIU\SFPR)+u_{i\bar{j}}\SFPR+u\SF_u=0.
\end{equation}
\item Multiply $z^q u_q$ on both sides of equation and use product rule twice, we get
\begin{equation}\label{calstep2}
\PZI(z^q u_q\PBZJ\SFPR)-\PBZJ\left( \PZI(z^q u_q)\SFPR \right)+u_{i\bar{j}}\SFPR+z^ku_{i\bar{j}k}\SFPR+z^q u_q\SF_u=0.
\end{equation}
\item
Use product rule and chain rule, we get
\begin{equation}\label{calstep3}
\frac{\partial}{\partial z^i}(z^i\SF)-n\SF=z^q\frac{\partial}{\partial z^q}\SF=z^q\SF_{z^q}+z^q u_q \SF_u +z^q u_{i\bar{j}q} \SFPR .
\end{equation}
\item
Multiplying \eqref{calstep1} by constant $c$ and combine it with \eqref{calstep2} and \eqref{calstep3}, we immediately get \eqref{PohId}.
\end{itemize}
\end{proof}

The relevant example to us is when
\begin{equation}\label{HSF}
\SF=-\frac{u S_k(u_{i\bar{j}})}{k+1}+F(z,u), \quad \mbox{ and } \mathbb{F}=\mathbb{F}_k=\mathbb{H}_k+\int_{\Omega} F(z,u)dV.
\end{equation}
where we define
\[
\mathbb{H}_k=-\frac{1}{k+1}\int_{\Omega}u S_k(u_{l\bar{m}})dV.
\]
The following lemma is well-known for the real k-Hessian operator (\cite{Rei}).  We give the complex version to see that \eqref{ELeq} in this case becomes the general complex k-Hessian equation
\begin{equation}\label{cHDi}
\left\{
\begin{array}{lr}
S_k(u_{l\bar{m}})=f(z,u),& \mbox{ on } \Omega,\\
u=0, &\mbox{ on } \partial\Omega.
\end{array}
\right.
\end{equation}
\begin{lem}\label{Newton}
Define the Newton tensor
\[
T_{k-1}(u_{l\bar{m}})^{i\bar{j}}=\frac{1}{k!}\sum_{}\delta^{i_1\dots i_{k-1}i}_{j_1\dots j_{k-1}j}u_{i_1\bar{j}_1}\dots u_{i_{k-1}\bar{j}_{k-1}}.
\]
Then we have
\begin{enumerate}
\item The tensor $\left(T_{k-1}(u_{l\bar{m}})^{i\bar{j}}\right)$ is divergence free, i.e.
\[
\PZI T_{k-1}(u_{l\bar{m}})^{i\bar{j}}=0=\PBZJ T_{k-1}(u_{l\bar{m}})^{i\bar{j}}
\]
\item
\[
S_k(u_{l\bar{m}})=\frac{1}{k}T_{k-1}(u_{})^{i\bar{j}}u_{i\bar{j}}.
\]
\item
\[
\frac{\partial S_k(u_{l\bar{m}})}{\partial u_{i\bar{j}}}=T_{k-1}(u_{l\bar{m}})^{i\bar{j}}.
\]
\end{enumerate}
\end{lem}
For the complex Hessian equation, we substitute \eqref{HSF} into \eqref{PohId} and use lemma \eqref{Newton} to get
\begin{eqnarray}\label{PohId}
&&\PZI\left(z^i\left(\frac{-u S_k(u_{l\bar{m}})}{k+1}+F(z,u)\right)+ (cu+z^qu_q)\frac{-u_{\bar{j}} T_{k-1}(u_{l\bar{m}})^{i\bar{j}}}{k+1}\right)\nonumber\\
&&+\PBZJ\left(\PZI(cu+z^qu_q)\frac{u T_{k-1}(u_{l\bar{m}})^{i\bar{j}}}{k+1}\right)\nonumber\\
&=&[k(c+1)+c-n]\frac{u S_k(u_{l\bar{m}})}{k+1}+nF-cu f+z^i F_{z^i}.
\end{eqnarray}
If we make the coefficient of the first term vanish, we get the important constant which will be useful later:
\[
c_0=\frac{n-k}{k+1}.
\]
The following  lemma is just the divergence theorem in complex coordinate. Note that we use the following standard normalizations.
\begin{equation}
\PZI=\frac{1}{2}\left(\frac{\partial}{\partial x^{2i-1}}-\sqrt{-1}\frac{\partial}{\partial x^{2i}}\right),\quad g_{i\bar{j}}=\frac{1}{2}\delta_{ij}, \quad \nu_{i}=g_{i\bar{j}}\nu^{\bar{j}}=\frac{1}{2}\nu^{\bar{i}},\quad
z^i\nu_i+\bar{z}^i\nu_{\bar{i}}=x^\alpha \nu_\alpha.
\end{equation}
\begin{lem}\label{cdivthm}
$\Omega$ is a bounded domain in $\mathbb{C}^n$ with $C^2$ boundary. Let $X=X^i\PZI$ be a $C^1$ vector field on $\bar{B}_1$ of type (1,0). Let $\nu$ denote the outward unit normal vector 
of $\partial\Omega$. Decompose $\nu=\nu^{(1,0)}+\nu^{(0,1)}$ such that $\nu^{(1,0)}=\nu^i\PZI$ and $\nu^{(0,1)}=\nu^{\bar{j}}\frac{\partial}{\partial\bar{z}^j}$. Then we have
\[ 
\int_{\Omega}\frac{\partial X^i}{\partial z^i}dV=\oint_{\partial\Omega} X^i \nu_i d\sigma, 
\]
where $d\sigma$ is the induced volume form on $\partial\Omega$ from the Euclidean volume form on $\mathbb{C}^n=\mathbb{R}^{2n}$.
\end{lem}

Assume $\Omega$ is a 
with $C^2$-boundary. For any $p\in\partial\Omega$, choose a small ball $B_{\epsilon}(p)$ such that $\Omega\cap B_{\epsilon}=\{\rho\le 0\}$, where $\rho$ is a $C^2$-function satisfying $|\nabla \rho|(p)=1$. Recall that the Levi form can be defined as
\[
\mathbb{L}=\sqrt{-1}\frac{\partial^2 \rho}{\partial z^i\partial\bar{z}^j}dz^i\otimes d\bar{z}^j.
\]
$\mathbb{L}$ is a symmetric Hermitian form on the space $\mathcal{T}=T^{(1,0)}\mathbb{C}^n\cap T(\partial\Omega)\otimes_{\mathbb{R}}\mathbb{C}=\{\xi\in\mathbb{C}^n; \xi^i f_i=0\}\cong (T(\partial\Omega)\cap J T(\partial\Omega), J)$, where $J$ is the standard complex structure on $\mathbb{C}^n\cong\mathbb{R}^{2n}$. Assume $\nu$ is the outer unit normal vector to $\partial\Omega$ then at point $p$, we have $\nu_i=\rho_i$. Denote
\begin{equation}\label{Levicurv}
\tilde{S}_{k-1}(\partial\Omega)=\frac{1}{(k-1)!}\sum_{1\le i_1,\dots, j_k\le n}\delta^{i_1\dots i_{k}}_{j_1\dots j_{k}}\rho_{i_1\bar{j}_1}\dots \rho_{i_{k-1}\bar{j}_{k-1}}\nu_{i_k}\nu_{\bar{j}_k}
=T_{k-1}(\rho_{l\bar{m}})^{i\bar{j}}\nu_i\nu_{\bar{j}}.
\end{equation}
We can choose coordinates, such that $\nu=\partial_{z^n}+\partial_{\bar{z}^n}$ and so $\nu_i=\frac{1}{2}\delta_{in}=\frac{1}{2}\nu^{\bar{i}}$. Then we see that, up to a constant, $\tilde{S}_{k-1}(\partial\Omega)$ is equal to $S_{k-1}\left(\mathbb{L}|_{\mathcal{T}}\right)$, the later being the $(k-1)$-th symmetric function of the eigenvalues of the restricted operator $\mathbb{L}|_{\mathcal{T}}$. 

Note that $\tilde{S}_{k-1}(\partial\Omega)$ is a well defined local invariant for $\partial\Omega$, i.e. it is independent of the defining function $\rho$. $\Omega$ is called to be strongly $k$-pseudoconvex, if $\tilde{S}_{k-1}(\partial\Omega)>0$
Note that the real version of $\tilde{S}_{k-1}(\partial\Omega)$ appeared in \cite[formula (6)]{Tso2}.

For example, when $\Omega$ is a ball $B_{R}(0)$,  $\nu_i=\frac{z^i}{2R}$ and we can choose $\rho=\frac{1}{2R}(|z|^2-R^2)$. By symmetry, we can calculate at point $(0,\cdots,0,1)$ to get: 
\begin{eqnarray}\label{tildeSball}
\tilde{S}_{k-1}(\partial B_R)&=&\frac{1}{(k-1)!}\sum_{i_1,\dots, j_{k-1}}\delta^{i_1\dots i_{k-1}n}_{j_1\dots j_{k-1}n} \rho_{i_1\bar{j}_1}\dots \rho_{i_{k-1}\bar{j}_{k-1}}\nu_i\nu_{\bar{j}}\nonumber\\
&=&\frac{1}{4R^2}\frac{1}{(2R)^{k-1}(k-1)!}\sum_{1\le i_1,\dots, j_{k-1}\le n-1}\delta^{i_1\dots i_{k-1}}_{j_1\dots j_{k-1}} \delta_{i_1\bar{j}_1}\dots \delta_{i_{k-1}\bar{j}_{k-1}}\nonumber\\
&=&\frac{1}{2^{k+1}R^{k+1}}\binom{n-1}{k-1}.
\end{eqnarray}

We can now derive the important integral formula for us. 
\begin{prop}
Let $\Omega$ be a $C^2$-domain. Suppose $f$ belongs $C(\bar{\Omega}\times (-\infty,0])\cap C^1(\Omega\times (-\infty, 0))$ and is positive in $\Omega\times
(-\infty,0)$. Assume $u\in C^2(\bar{\Omega})\cap C^4(\Omega)$ is a solution to \eqref{cHDi}. Then we have the identity
\begin{equation}\label{PohIdentity1}
\oint_{\partial\Omega} z^i\nu_i \tilde{S}_{k-1}(\partial\Omega) |\nabla u|^{k+1} d\sigma=-(k+1)\int_{\Omega} \left(n F-\frac{n-k}{k+1}u f+z^i F_{z^i}\right)dV.
\end{equation}
\begin{equation}\label{PohIdentity2}
\oint_{\partial\Omega} \langle x,\nu\rangle \tilde{S}_{k-1}(\partial\Omega) |\nabla u|^{k+1} d\sigma=-(k+1)\int_{\Omega} \left(2(n F-\frac{n-k}{k+1}u f)+x^{\alpha} F_{x^{\alpha}}\right)dV.
\end{equation}
\end{prop}
\begin{proof}
The second identity follows from the first easily. So we only prove the first identity. When $\nabla u\neq 0$, letting $\rho=\frac{u}{|\nabla u|}$ in \eqref{Levicurv}, we get
\begin{eqnarray*}
z^k u_k u_{\bar{j}}T_{k-1}(u_{p\bar{q}})^{i\bar{j}}\nu_i&=&
z^k \nu_k  |\nabla u|^{-(k-1)} \nu_{\bar{j}}T_{k-1}(u_{p\bar{q}})^{i\bar{j}}\nu_i|\nabla u|^{k+1}\\
&=&z^k\nu_k \tilde{S}_{k-1} |\nabla u|^{k+1}.
\end{eqnarray*}
When $\nabla u=0$, then both sides are equal to zero.
Now we can integrate \eqref{PohId} on $B_1$ using divergence theorem (Lemma \ref{cdivthm}) and the boundary condition $u=0$ on $\partial B_1$ to get the first identity.
\end{proof}
\begin{proof}[Part I of Proof of Theorem \ref{critCHE}]
When $\Omega=B_R$, $\langle x,\nu\rangle=R>0$. $\tilde{S}_{k-1}(\partial B_R)\ge 0$ is a positive constant. So the left hand side of \eqref{PohIdentity2} is positive. When $f(u)=(-u)^{p}$, $F(u)=-\frac{1}{p+1}(-u)^{p+1}$. So if
$\frac{n}{p+1}-\frac{n-k}{k+1}\le 0$, i.e. $p\ge \frac{(n+1)k}{n-k}=\gamma(k,n)$, there is no nontrivial nonpositive solution in $C^2(\bar{B}_1)\cap C^4(B_1)$ to \eqref{chess}.
\end{proof}
\begin{rem}
Similar argument actually gives non-existence result for star-shaped and strongly $k$-pseudoconvex domains.
\end{rem}

\section{Non-local problem with exponential nonlinearities}
In this section, we prove Theorem \ref{nonlocalthm} using Pho\v{z}aev method. As mentioned before, when $k=1$ the argument was used in \cite{CLMP}. The argument can be generalized to higher $k$ by the introduction of $\tilde{S}_{k-1}(\partial\Omega)$ in \eqref{Levicurv}. Recall that we consider the following non-local equation:
\begin{equation}\label{nonlocaleq2}
\det(u_{l\bar{m}})=a\frac{e^{-u}}{\int_{B_1} e^{-u}dV},\quad u=0 \mbox{ on } \partial B_1.
\end{equation}
\begin{proof}
In identity \eqref{PohIdentity2}, if $f$ does not depend on $z$, then it becomes:
\begin{equation}\label{PohIdentity3}
-2 \int_{\Omega} \left(n (k+1) F(u)-(n-k)u f(u)\right)dV=\int_{\partial\Omega}\langle x,\nu\rangle\tilde{S}_{k-1}(\partial\Omega)|\nabla u|^{k+1}d\sigma.
\end{equation}
To estimate the right hand side, note that we can integrate both sides of \eqref{cHLBG} and use divergence theorem to get 
\begin{eqnarray*}
a&=&\int_{\Omega} S_k(u_{l\bar{m}})=\frac{1}{k}\int_{\Omega}T_{k-1}(u_{l\bar{m}})^{i\bar{j}}u_{i\bar{j}}dV\\
&=&\frac{1}{k}\oint_{\partial \Omega} u_iT_{k-1}(u_{l\bar{m}})^{i\bar{j}} (\nu^{p}g_{p\bar{j}})= \frac{1}{k}\oint_{\partial \Omega} \tilde{S}_{k-1}(\partial \Omega) |\nabla u|^k.
\end{eqnarray*}
For simplicity, we let $\tilde{S}_{k-1}$ denote the quantity $\tilde{S}_{k-1}(\partial\Omega)$ defined in \eqref{Levicurv}. Now by H\"{o}lder's inequality, we have
\begin{eqnarray*}
ka&=&\oint_{\partial\Omega} \tilde{S}_{k-1}|\nabla u|^k=\oint_{\partial\Omega}(\langle x,\nu\rangle\tilde{S}_{k-1})^{k/(k+1)}|\nabla u|^k  (\langle x,\nu\rangle)^{-k/(k+1)}\tilde{S}_{k-1}^{1/(k+1)}\\
&\le &\left(\oint_{\partial\Omega}\langle x,\nu\rangle\tilde{S}_{k-1} |\nabla u|^{k+1}\right)^{k/(k+1)}\left(\oint_{\partial\Omega}\langle x,\nu\rangle^{-k}\tilde{S}_{k-1}\right)^{1/(k+1)}
\end{eqnarray*}
So we get
\begin{equation}\label{starlowerbound1}
\oint_{\partial\Omega} \tilde{S}_{k-1}|\nabla u|^{k+1}\ge \frac{(ka)^{(k+1)/k}}{\left(\oint_{\partial\Omega}\langle x,\nu\rangle^{-k} \tilde{S}_{k-1}\right)^{1/k}}.
\end{equation}
Now we specialize to equation \eqref{nonlocaleq2}. When $\Omega$ is the unit ball, $\langle x,\nu\rangle\equiv 1$ and, by \eqref{tildeSball} and for simplicity, we denote
\[
\tilde{S}_{k-1}=\tilde{S}_{k-1}(\partial B_1)\equiv\frac{1}{2^{k+1}}\binom{n-1}{k-1}.
\]  
Also we have
\[
f(u)=a\frac{e^{-u}}{\int_{B_1}e^{-u}dV},\quad F(u)=a\frac{1-e^{-u}}{\int_{B_1}e^{-u}dV}.
\]
Combine \eqref{PohIdentity3} and \eqref{starlowerbound1}, we get
\[
2a\left(\tilde{S}_{k-1}|\partial B_1|\right)^{1/k}\int_{B_1}[(n-k)u e^{-u}+n(k+1)(e^{-u}-1)] dV\ge (k a)^{(k+1)/k} \int_{B_1} e^{-u} dV.
\]
So there is no solution if $a$ satisfies
\begin{equation}
a\ge \frac{(2n(k+1))^k\omega_{2n-1} \tilde{S}_{k-1}}{k^{k+1}}=\left(\frac{n(k+1)}{k}\right)^k\binom{n}{k}\frac{\pi^n}{n!}=:\alpha_1(k,n).
\end{equation}
When $k=n$, the righthand is equal to $(n+1)^n\pi^n/n!$ which is sharp.
\end{proof}
\begin{rem}
When $k<n$, we can get better estimate for $a$. For this, consider the function
\[
\mu(x)=c_1 (e^x-1)-c_2 xe^x-c_3 e^x. 
\]
with $c_1=n(k+1)$, $c_2=(n-k)$ and $c_3=k^{(k+1)/k}\alpha_2(k,n)^{1/k}(2\tilde{S}_{k-1}\omega_{2n-1})^{-1/k}$. The condition 
${\rm max}\{\mu(x); x\ge 0\}=0$
gives a better upper bound $\alpha_2(k,n)$ for $a$, although it's still not sharp:
\begin{eqnarray*}
0<\alpha_2(k,n)
=\alpha_1(k,n)\left[1-\frac{n-k}{n(k+1)}+\frac{n-k}{n(k+1)}\log\frac{n-k}{n(k+1)}\right]^k\le \alpha_1(k,n).
\end{eqnarray*}
\end{rem}
\begin{rem}
If we consider the similar real Hessian equation on $B_1\subset \mathbb{R}^{d}$:
\begin{equation}\label{rHLBG}
S_{k}(u_{\alpha\beta})=\tilde{a}\frac{e^{-u}}{\int_{B_1}e^{-u}dV} \mbox{ on } B_1, \quad u=0 \mbox{ on } \partial B_1,
\end{equation}
then we can use the real version of above calculation to get the following necessary condition for $a$ in order for \eqref{rHLBG} to have a solution in $C^2(\bar{B}_1)\cap C^4(B_1)$. 
\[
\tilde{a}<\tilde{\alpha}(k,d):=\frac{((k+1)d)^k \binom{d-1}{k-1}\omega_{d-1}}{k^{k+1}}.
\] 
The case when this bound is sharp is when the real dimension is even $d=2n$ and $k=d/2=n$. Indeed, we have
\begin{prop}
When $k=\frac{d}{2}$, then there exists a solution in $C^2(\bar{B}_1)\cap C^4(B_1)$ to \eqref{rHLBG} if and only if $\tilde{a}<\tilde{\alpha}(d/2,d)$.
\end{prop}
\begin{proof}
We just need to show that, for $k=d/2$, there exists a radially symmetric solution for \eqref{rHLBG} when $a<\tilde{\alpha}(k,d)$.
First it's easy to verify that the radial symmetry reduces the equation \eqref{rHLBG} to the following equation:
\[
\frac{d-2k}{d}\binom{d}{k}(u_rr)^k+\frac{1}{k}\binom{d-1}{k-1}(u_rr)^{k-1} (u_{r}r)_rr=\frac{\tilde{a}}{\omega_{d-1}}\frac{r^{2k} e^{-u}}{\int_0^1 e^{-u(r)}r^{2n-1}dr}. \quad u=0 \mbox{ on } \partial B_1.
\]
Now assume $k=\frac{d}{2}=n$ and we introduce the variable $s=r^2$. Then the above equation becomes:
\begin{equation}\label{realhalfrad}
((u_ss)^n)_ss=\frac{n^2\tilde{a}}{\omega_{d-1}\binom{d-1}{n-1}2^{n}}\frac{s^n e^{-u}}{\int_0^1 e^{-u(s)}s^{n-1}ds}.
\end{equation}
This equation is integrable since it's the same as the radial reduction of complex complex Monge-Amp\`{e}re equation. See \eqref{radLBGEQ}, \eqref{FSsol} and \eqref{aepsilon}.
So it has solution
\[
u_{\epsilon}=(n+1)[\log(|x|^2+\epsilon^2)-\log (1+\epsilon^2)],
\]
with the parameter
\[
\tilde{a}_{\epsilon}=\frac{1}{n}\binom{d-1}{n-1}2^{n+1}a_\epsilon=\frac{1}{n}\binom{d-1}{n-1}(2n+2)^n\frac{\omega_{d-1}}{(1+\epsilon^2)^n}.
\]
So $\tilde{a}_\epsilon\in (0,\tilde{\alpha}(d/2,d)=\frac{2}{d}\binom{d-1}{d/2-1}(d+2)^{d/2}\omega_{d-1})$.
\end{proof}
From another point of view, in \cite{TiWa}, Tian-Wang proved the following Moser-Trudinger inequality for $k=d/2$:
\[
\int_{\Omega} \exp \left(D\left(\frac{u}{\|u\|_{\Phi_0^k}}\right)^{p_0} \right)\le C.
\]
with
\[
\|u\|_{\Phi_0^k}=\left(\int_{\Omega}-u S_{k}(u_{\alpha\beta})\right)^{1/(k+1)}.
\]
\[
D=d\left[\frac{\omega_{d-1}}{k}\binom{d-1}{k-1}\right]^{2/d},\quad p_0=\frac{d+2}{d}.
\]
If we let $x=u/\|u\|_{\Phi_0^k}$ and $y=\|u\|_{\Phi_0^k}$ and use the inequality 
\[
xy\le D x^{p_0}+ E y^{q_0}, \mbox{ with } q_0=\frac{d}{2}+1, E=(D p_0)^{-q_0/p_0}q_0^{-1}=\left[(d+2)^{d/2}\frac{\omega_{d-1}}{k}\binom{d-1}{k-1}\frac{d+2}{2}\right]^{-1}.
\]
we get the Moser-Trudinger-Onofri inequality:
\[
-(E(d/2+1))^{-1}\log\left(\int_{\Omega}\exp(-u) dV\right)\le \frac{1}{k+1}\int_{\Omega}-u S_{d/2}(u_{\alpha\beta})dV+C.
\]
This implies when $0<a<E(k+1)^{-1}$, there exists a solution to \eqref{rHLBG}. Now note that we indeed have: (k=d/2)
\[
\tilde{\alpha}(d/2)=(E(k+1))^{-1}=(d+2)^{d/2}\frac{2}{d}\binom{d-1}{k-1}\omega_{d-1}.
\]
\end{rem}

\section{Radially symmetric solutions}
\subsection{Reduction in the radially symmetric case}

In this section, we assume $\Omega=B_R$ and $u(z)=u(s)$ is radially symmetric, where $s=r^2=|z|^2$. Then we can calculate that 
\[
u_{i\bar{j}}=u_s\delta_{ij}+U_{ss} \bar{z}^i z^j.
\]
By the unitary invariance of operator $S_k$, we get
\begin{eqnarray*}\label{radSk}
S_k(u_{l\bar{m}})&=&\binom{n-1}{k}u_s^k+\binom{n-1}{k-1}u_s^{k-1}(u_s+u_{ss}s)\\
&=&\frac{1}{k}\binom{n-1}{k-1}(u_s^k s^n)_ss^{1-n}.
\end{eqnarray*}
So the radially symmetric solution to \eqref{chess} satisfies the equation:
\begin{equation}\label{radchess}
\frac{1}{k}\binom{n-1}{k-1}(u_s^k s^n)_ss^{1-n}=(-u)^p, \quad u(R)=0.
\end{equation}
The Hessian energy becomes
\begin{eqnarray*}
\mathbb{H}_k=-\frac{1}{k+1}\int_{\Omega}u S_k(u_{l\bar{m}})dV=\frac{\omega_{2n-1}}{2k(k+1)}\binom{n-1}{k-1}\int_0^R u_s^{k+1} s^nds.
\end{eqnarray*}
so the functional whose Euler-Lagrange equation is \eqref{radchess} becomes
\[
\mathbb{F}_k=\frac{A}{k+1} \int_0^R |u_s|^{k+1} s^n ds-\frac{B}{p+1}\int_0^R |u|^{p+1}s^{n-1}ds
\]
where
\begin{equation}\label{defAB}
A=A(k,n)=\frac{\omega_{2n-1}}{2k}\binom{n-1}{k-1}, \quad B=B(k,n)=\frac{\omega_{2n-1}}{2}.
\end{equation}
As in \cite{Tso1}, denote $
\mathcal{E}=\{u\in C^1([0,R]); u(R)=0\}$. For any $1\le k\le n$ and $0<\delta<\gamma(k,n)=\frac{(n+1)k}{n-k}$, 
and let $\mathcal{W}_k$ be the completion of $\mathcal{E}$ under the norm
\[
\|u\|=\left(\int_0^R u_s^{k+1} s^{n} ds\right)^{1/(k+1)}.
\]
\begin{lem}
There exists a constant $C=C(\delta, k, R, n)$ such that, for all $u\in E$,
\[
\left(\int_0^R |u|^{\delta+1} s^{n-1} ds\right)^{1/(\delta+1)}\le C \left(\int_0^R |u_s|^{k+1} s^n\right)^{1/(k+1)}.
\]
\end{lem}
\begin{proof}
By applying H\"{o}lder's inequality to $u(s)=\int_R^s u_s(s)ds$, we have
\[
|u(s)|\le C s^{-(n-k)/(k+1)}\left(\int_0^R |u_s|^{k+1} s^n\right)^{1/(k+1)}.
\]
Then raising the $(\delta+1)$-th power, multiplying $s^{n-1}$ and integrating from $0$ to $R$ we get the inequality. The range for $\delta$ is determined by the inequality:
\[
-\frac{n-k}{k+1}(\delta+1)+n-1>-1.
\]
\end{proof}
\begin{rem}
By \cite{CKN}, when $k<n$, we actually have the sharp Sobolev inequalities of complex Hessian operator for radial functions,
\[
\left(\int_0^R |u|^{\gamma(k,n)+1} s^{n-1}\right)^{1/(\gamma(k,n)+1)}\le C\left(\int_0^R |u_s|^{k+1}s^n \right)^{1/(k+1)}.
\]
Since we don't have symmetrization process as in the real case, the sharp Sobolev inequalities for general k-plurisubharmonic functions are open (\cite{Zhou}).
\end{rem}
As in \cite{Tso1}, we define the notion of weak solution. We use the constants in \eqref{defAB}.
\begin{defn}
We say $u\in\mathcal{W}_k$ is a weak solution to equation \eqref{radchess}, if for every $\phi\in C^1([0,R])$ with $\phi(R)=0$, the following identity is satisfied.
\[
A\int_0^R |u_s|^{k} u_s \phi'(s) s^nds=B\int_0^R |u|^p \phi(s) s^{n-1}ds. 
\] 
\end{defn}
Arguing as in \cite[Lemma 4 ]{Tso1}, we get the following regularity result which reduces the problem to finding critical point of $\mathbb{F}_k$ on $\mathcal{W}_k$.
\begin{lem}[\cite{Tso1}]
Any generalized solution of \eqref{radchess} is in $C^2([0,R])$, and solves \eqref{radchess} in the classical sense. Moreover, it is negative in $[0,R)$ unless it
vanishes identically.
\end{lem}
\begin{proof}[Part II of Proof of Theorem \ref{critCHE}]
When $p<k$, we are in the sub-linear (with respect to complex k-Hessian operator) case, by the Sobolev inequality, we have
we have
\[
\int_0^R |u|^{p+1}s^{n-1}ds \le C(p)\left(\int_{0}^R |u_s|^{k+1} s^n ds\right)^{(p+1)/(k+1)}\le \epsilon\int_0^R |u_s|^{k+1}s^n ds+C(\epsilon,p).
\]
Then by taking $\epsilon$ sufficiently small, we get
\[
\mathbb{F}_k\ge \epsilon \int_0^R|u_s|^{k+1}s^nds -C(\epsilon, p).
\]
So 
the functional $\mathbb{F}_k$ is a coercive functional on $\mathcal{W}_k$ and one can use the direct method in variational calculus to find an absolute minimizer. On the other hand,
it's easy to see that
\[
\mathbb{F}_k(tu)=O(t^{k+1})-O(t^{p+1})< 0, \mbox{ as } t\ll 1.
\]
So the absolute minimizer is not $0$.

In the super-linear case, i.e. when $k<p<\gamma(k)$, we have
\begin{enumerate}
\item
$\mathbb{F}_k(0)=0$, and $\mathbb{F}_k(tu)=O(t^{k+1})-O(t^{p+1})\rightarrow -\infty$ as $t\rightarrow +\infty$.
\item Choose $\alpha$ sufficiently small, then when $\|u\|=\alpha$
\[
\mathbb{F}_k(u)\ge \|u\|-C(p)\|u\|^{(p+1)/(k+1)}=\|u\|\left(1-C(p)\|u\|^{\frac{p-k}{k+1}}\right)=\alpha\left(1-C(p)\alpha^{\frac{p-k}{k+1}}\right)>0.
\]
\end{enumerate}
So $\mathbb{F}_k$ satisfies the Montain Pass condition. Now as in the semi-linear case, it's known that under the assumption, $\mathbb{F}_k$ is in $C^1(\mathcal{W}_k,\mathbb{R})$ and satisfies the Palais-Smale condition. So the minimax method proves the existence of critical point of $\mathbb{F}_k$ on $\mathcal{W}_k$. For details, see \cite{Rab}.  
\end{proof}

\subsection{Nonlocal problem with exponential nonlinearity}

Denote $s=|z|^2$. Assume $u=u(s)$ is any radial symmetric solution of \eqref{cHLBG}. Then by \eqref{radSk}, we see that \eqref{cHLBG} is reduced to 
the following equation for $u$:
\begin{equation}\label{radLBGEQ}
(u_s^ks^n)_ss^{1-n}=\lambda e^{-u}, \quad \lambda=\frac{2k}{\binom{n-1}{k-1}\omega_{2n-1}}\frac{a}{\int_{0}^1e^{-u(s)}s^{n-1}ds}=A(k,n)^{-1}\frac{a}{\int_0^1 e^{-u(s)}s^{n-1}ds}.
\end{equation}
We use the phase plane method to study this equation. Define
\[
v=\left(\frac{1}{k}u_ss\right)^k,\quad w=\lambda k^{-k} s^k e^{-u}.
\]
Introduce a new variable $t=\log s$. Then it's easy to verify \eqref{radLBGEQ} is equivalent to the following system of equations:
\begin{equation}\label{ODEsys}
v_{t}=-(n-k)v+w,\quad w_{t}=kw(1-v^{1/k}).
\end{equation}
For the boundary condition, when $r=-\infty$, or equivalently $s=0$.
\[
v(-\infty)=0=w(-\infty). 
\]
To find the boundary condition when $t=0$, or equivalently $s=1$, we note that
\[
\int_{B_{\sqrt{s}}}\det(u_{l\bar{m}})dV=\frac{1}{k}\binom{n-1}{k-1}\frac{\omega_{2n-1}}{2}\int_0^{s}(u_s^ks^n)_s ds=A(k,n) u_s^k s^n.
\]
So 
\[
v(t=0)=k^{-k}A(k,n)^{-1}\int_{B_1}\det(u_{l\bar{m}})dV=k^{-k}A(k,n)^{-1}a.
\]
while $w(t=0)=\lambda k^{-k}$.  So we are looking for the trajectory from $(0,0)$ to the point $(k^{-k}A(k,n)^{-1}a,\lambda k^{-k})$. The critical point of 
system \eqref{ODEsys} is $(1, (n-k))$. The Hessian matrix is
\[
\left.\left(\begin{array}{cc}
-(n-k)& 1\\
-w v^{(1-k)/k}&k(1-v^{1/k})
\end{array}\right)\right|_{(1,(n-k))}=
\left(\begin{array}{cc}
k-n& 1\\ -(n-k)&0
\end{array}\right).
\]
whose trace and determinant are
\[
{\rm tr}=k-n,\quad \det=n-k.
\]
So the two eigenvalue is
\[
\beta_1=\frac{k-n+ \sqrt{(n-k)^2-4(n-k)}}{2}, \quad \beta_2=\frac{k-n- \sqrt{(n-k)^2-4(n-k)}}{2}.
\]
There are two complex eigenvalue with negative real part if and only if
\[
0<n-k<4.
\]
Now we can prove Theorem \ref{RadExist} using similar analysis as in \cite{BHN} (see also \cite{JoLu} and \cite{JaSc}). 
\begin{proof}[Proof of Theorem \ref{RadExist}]
When $n=k$, the equation is integrable. $u=(n+1)[\log (s+\epsilon^2)-\log(1+\epsilon^2)]$. 
\[
v(s)=\left(\frac{1}{n}u_s s\right)^k=\left(\frac{n+1}{n}\right)^{n}\left(\frac{s}{s+\epsilon^2}\right)^n, \quad w(s)=\frac{(n+1)^n}{n^{n-1}}\frac{\epsilon^2 s^n}{(s+\epsilon^2)^{n+1}}.
\]
So there is a trajectory $\mathcal{O}$ connecting $(0,0)$ to the point $((\frac{n+1}{n})^n,0)$ and  $a_{\epsilon}=n^n v(t=0) C(n,n)=(n+1)^n\frac{\pi^n}{n!(1+\epsilon^2)^n}$ lies in $(0,a_0=(n+1)^n\frac{\pi^n}{n!})$.

When $k<n$, consider the function defined by
\[
L(v,w)=k\left(\frac{k}{k+1}v^{(k+1)/k}-v+\frac{1}{k+1}\right)+(w-(n-k))-(n-k)\log\frac{w}{(n-k)}.
\]
Then it's easy to verify that $L(1,n-k)=0$ and $L(v,w)>0$ for $\mathbb{R}_+^2\ni (v,w)\neq (1,n-k)$. Moreover, if $(v(t),w(t))$ is a trajectory for the system \eqref{ODEsys},
then 
\[
\frac{d}{dt}L(v(t),w(t))=-(n-k)k(v^{1/k}-1)(v-1)\le 0, \mbox{ and } <0 \mbox{ when } v\neq 1.
\]
So $L(v,w)$ is a Lyapunov function for the system \eqref{ODEsys}. So we conclude that the basin of attraction of $(1,n-k)$ contains the whole positive quadrant. The solution to \eqref{radLBGEQ} corresponds to a trajectory $\tilde{\mathcal{O}}$ connecting $(0,0)$ to $(v(t=0),w(t=0))$.
\begin{enumerate}
\item
When $n-k<4$, ${\rm Im}(\beta_{1,2})\neq 0$ and ${\rm Re}(\beta_{1,2})<0$. There is a trajectory $\mathcal{O}$ connecting $(0,0)$ and $(1,n-k)$, which turns around $(1,n-k)$ infinitely many times. 
In particular, the line $v=1$ intersects with $\mathcal{O}$ at infinitely many points.  This behavior of $\mathcal{O}$ clearly implies part 2 of Theorem \ref{RadExist}.
\item
When $n-k\ge 4$, we consider the region $\mathcal{D}$ bounded by the curves $\mathcal{C}=\{w=(n-k)v^b\}$ and $w=(n-k)v$. 

{\bf Claim:} When $(-\beta_2)^{-1}\le b\le(-\beta_1)^{-2}$, the region is invariant under the system \eqref{ODEsys}. 

{\bf Proof of the claim:} We just need to show the vector field on the boundary of the region points to the interior of the region. For the boundary $w=(n-k)v$ this is clear since the vector field has direction $\langle 0,1\rangle$. For the boundary $w=(n-k)v^b$, we parametrize it by $\{v=\tau, w=(n-k)\tau^b; 0\le \tau\le 1\}$. For $0<\tau<1$, the vector field points to the interior if and only if
\begin{eqnarray*}\label{invcond}
&&\frac{kw(1-v^{1/k})}{-(n-k)v+w}=\frac{k(n-k)t^b(1-t^{1/k})}{-(n-k)\tau+(n-k)\tau^b}<b(n-k)\tau^{b-1}\nonumber\\
&\Longleftrightarrow& h(\tau):=k(1-\tau^{1/k})-b(n-k)(\tau^{b-1}-1)<0\nonumber.
\end{eqnarray*}
\[
h(0)=-\infty, h(1)=0, h'(\tau)=\tau^{b-2}((n-k)b(1-b)-\tau^{\frac{1}{k}+1-b}).
\]
So if $h(\tau)$ is increasing, i.e. $h'(\tau)>0$ when $\tau\in (0,1)$, then \eqref{invcond} holds. Now $h'(\tau)> h'(1)=(n-k)b(1-b)-1$. It's easy to see that
\[
h'(1)\ge 0 \Longleftrightarrow (-\beta_2)^{-1}\le b\le (-\beta_1)^{-1}.
\]
So we can just choose the curve $\mathcal{C}=\{w=(n-k)v^{-1/\beta_1}\}$. Now it's easy to see that $\mathcal{O}$ lies in the region $\mathcal{D}$. Since $\mathcal{D}$ is above the curve $w=(n-k)v$, so $v'(t)\ge 0$ along $\mathcal{O}$. This implies for any $0<v(t=0)\le 1$, or equivalently, when $0<a\le k^k C(n,k) v(t=0)=k^{k-1}\binom{n-1}{k-1}\frac{\pi^n}{(n-1)!}$, there exists a unique solution to \eqref{radLBGEQ}.
\end{enumerate}
\end{proof}
\begin{figure}[h]
  \begin{center}
  \subfigure[$n=6, k=6$]{\label{zerodif}\includegraphics[height=4.5cm]{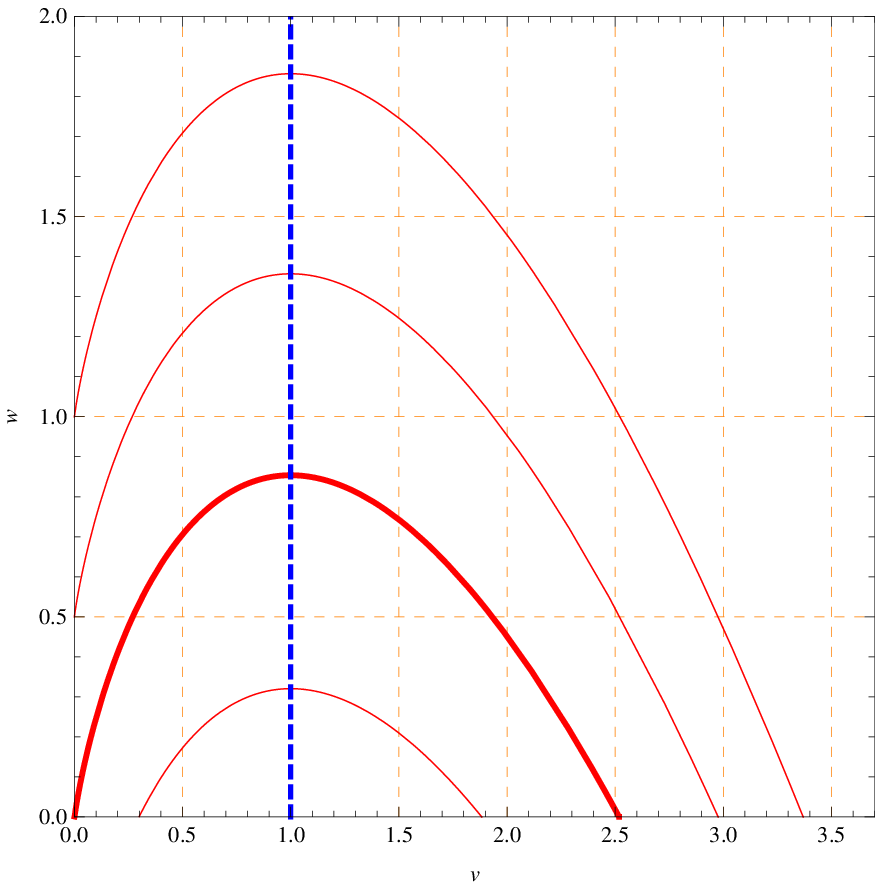}}
    \subfigure[$n=6,k=1$]{\label{bigdif}\includegraphics[height=4.5cm]{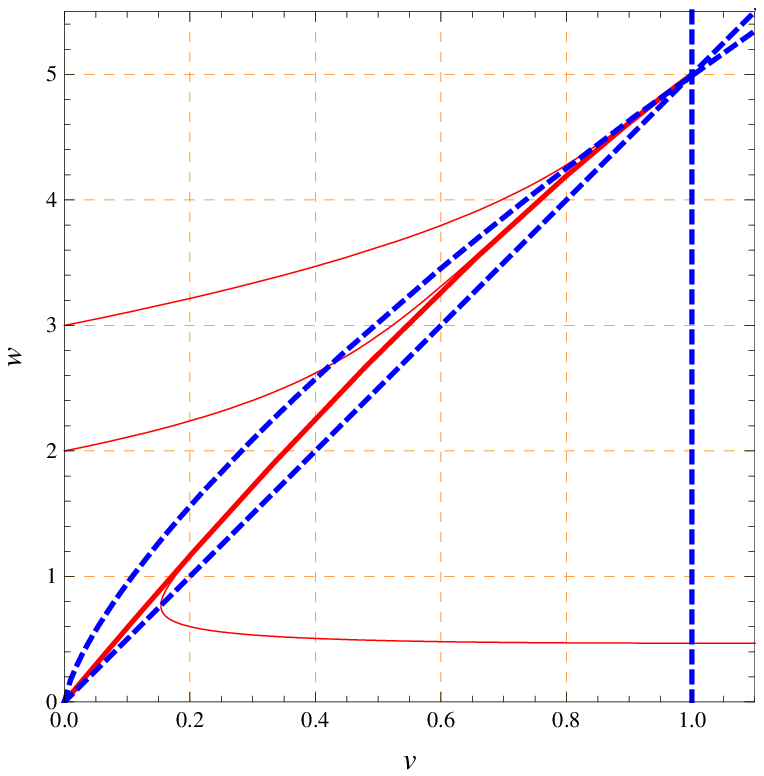}}
    \subfigure[$n=6,k=5$]{\label{smalldif}\includegraphics[height=4.5cm]{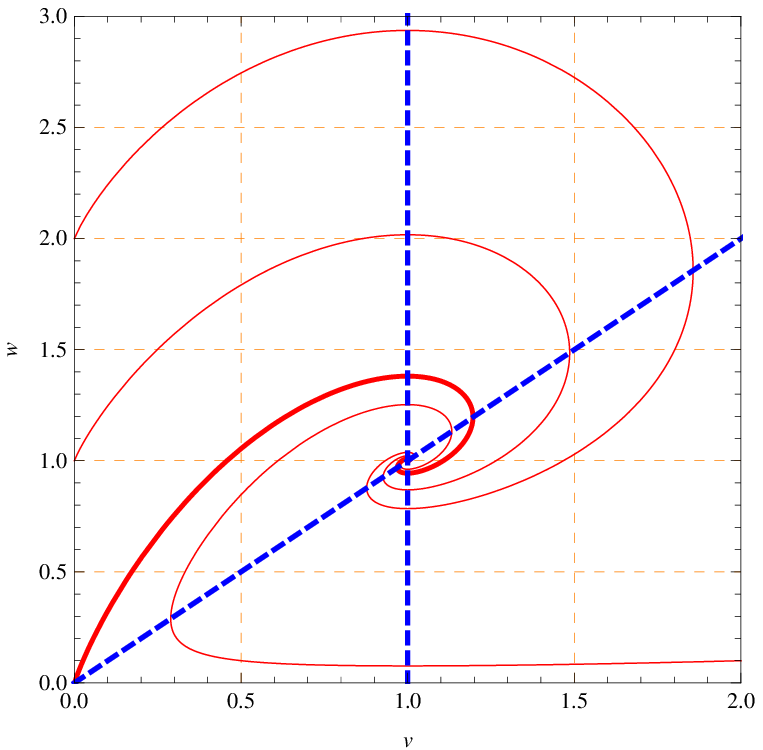}}   
  \end{center}
  \caption{Phase diagrams for system \eqref{ODEsys} when $n=6$}
  \label{phase}
\end{figure}
In figure \ref{phase}, we give phase diagrams in three cases of the above proof when $n=6$.
\begin{rem}\label{regionmodify}
In the case where $n-k\ge 4$, define the line $\mathcal{L}$ to be one characteristic line of the system: $(n-k)(v-1)+\beta_1(w-(n-k))=0$. Note that the curve $\mathcal{C}=\{w=v^{-1/\beta_1}\}$ is tangent to $\mathcal{L}$.
In \cite{BHN}, the region was chosen to be a triangle bounded by $\mathcal{L}$, $v=0$ and $w=(n-k)v$. But one can verify that, for some choices of $(n,k)$ for complex Hessian equation this triangle is not invariant under the flow. So it's more natural to consider the above invariant region $\mathcal{D}$ when one deals with general Hessian case. 
\end{rem}


\noindent
Mathematics Department, Stony Brook University, Stony Brook NY, 11794-3651, USA \\
Email: chi.li@stonybrook.edu

\end{document}